\documentclass{birkjour}
\usepackage[english]{babel}
\usepackage{amsmath,amssymb,amsthm,amscd}
\usepackage{paralist,prettyref} 
\usepackage{hyperref}
\usepackage{color}

\newcommand{\Blue}[1]{\color{blue}#1\color{black}}

\newtheorem{thm}{Theorem}[section]

\newtheorem{lem}[thm]{Lemma}

\theoremstyle{definition}

\theoremstyle{remark}

\numberwithin{equation}{section}

\newrefformat{cor}{Corollary~\ref{#1}}
\newrefformat{defn}{Definition~\ref{#1}}
\newrefformat{eq}{(\ref{#1})}
\newrefformat{lem}{Lemma~\ref{#1}}
\newrefformat{prob}{problem~(\ref{#1})}
\newrefformat{prop}{Proposition~\ref{#1}}
\newrefformat{thm}{Theorem~\ref{#1}}
\newrefformat{rem}{Remark~\ref{#1}}

\newcommand{\bbB}{\mathbb B}
\newcommand{\bbC}{\mathbb C}

\newcommand{\bbE}{\mathbb E}

\newcommand{\bbR}{\mathbb R}
\newcommand{\bbX}{\mathbb X}

\newcommand{\calA}{\mathcal A}
\newcommand{\calB}{\mathcal B}

\newcommand{\calL}{\mathcal L}

\newcommand{\del}{\partial}

\newcommand{\eps}{\varepsilon}

\newcommand{\0}{\mspace{0mu}_0}       
\renewcommand{\Re}{\operatorname{Re}}

\newcommand{\tr}{\operatorname{tr}}

\begin{document}

\title[Regularity and long-time behavior for the Westervelt equation]{Optimal regularity and long-time behavior of solutions for the Westervelt equation}

\author[S. Meyer]{Stefan Meyer}
\address{%
Martin-Luther-Universität Halle-Wittenberg\\
Naturwissenschaftliche Fakult\"{a}t II\\
Institut f\"{u}r Mathematik\\
06099 Halle (Saale)}
\email{stefan.meyer@mathematik.uni-halle.de}

\author[M. Wilke]{Mathias Wilke} 
\address{%
Martin-Luther-Universität Halle-Wittenberg\\
Naturwissenschaftliche Fakult\"{a}t II\\
Institut f\"{u}r Mathematik\\
06099 Halle (Saale)}
\email{mathias.wilke@mathematik.uni-halle.de}


\keywords{Westervelt equation, optimal regularity, quasilinear parabolic system, exponential stability\footnote{This manuscript is published in Appl. Math. Optim., \textbf{64} (2011), 257--271.  The final publication is available at Springer via \url{http://dx.doi.org/10.1007/s00245-011-9138-9}}}

\date{\today}
\dedicatory{Dedicated to Jan Pr\"{u}ss on the occasion of his 60th birthday.}

\begin{abstract}
We investigate an initial-boundary value problem for the quasilinear Westervelt equation which models the propagation of sound in fluidic media. We prove that, if the initial data are sufficiently small and regular, then there exists a unique global solution with optimal $L_p$-regularity. We show furthermore that the solution converges to zero at an exponential rate as time tends to infinity. Our techniques are based on maximal $L_p$-regularity for abstract quasilinear parabolic equations.
\end{abstract}

\maketitle

\section{Introduction and Notations}

The aim of this paper is to enhance the mathematical understanding of the Westervelt equation, which was proposed and analyzed in \cite{ClKaVe09,Kal10,KaLa09}:
\begin{align}\label{eq:Westervelt_1}
  u'_{tt} - c^2\Delta_x u' - b \Delta_x u'_t &= k (u'^2)_{tt}.
\end{align}
Here $u'(t,x)=u(t,x)-u_0$ denotes the acoustic pressure fluctuation at time $t$ and position $x$. Furthermore, $c>0$ denotes the velocity of sound, $b>0$ the diffusivity of sound and $k>0$ the parameter of nonlinearity. This equation is used to describe the propagation of sound in fluid media. It can be derived from the balances of mass and momentum (the compressible Navier-Stokes equations for Newtonian fluids) and a state equation for the pressure-dependent density of the fluid. A generalization of \eqref{eq:Westervelt_1} is given by Kuznetsov's equation:
\begin{align}\label{eq:Kuznetsov}
  u'_{tt} - c^2\Delta_x u' - b \Delta_x u'_t &= k (u'^2)_{tt}+|\nabla v'|^2_{tt}.
\end{align}
Here the velocity fluctuation $v'(t,x)=v(t,x)-v_0$ is related to the pressure fluctuation by means of an acoustic potential $\psi(t,x)$, such that $u'=\rho_0\psi_t$, $v'=-\nabla\psi$. This equation is used as a basic equation in nonlinear acoustics, see \cite{HaBl98,Kal07,LeSeWo08}. We refer to \cite{Kal07} for a derivation of Kuznetsov's equation.

To avoid any confusion, we mention that in the engineering literature \cite{HaBl98,Kal07,LeSeWo08} also the following equation is termed Westervelt equation.
\begin{align*}
  u'_{tt} - c^2\Delta_x u' + \frac{b}{c^2} u'_{ttt} &= k (u'^2)_{tt}.
\end{align*}
This equation is derived in \cite{HaBl98}, for instance. However, we may still view \eqref{eq:Westervelt_1} as a simplification of \eqref{eq:Kuznetsov}.

Let $\Omega\subset\bbR^n$ be a bounded $C^2$-domain, let $J=(0,T)$, $T\in(0,\infty]$ and let $u_0,u_1$ be given functions on $\Omega$. Then we consider the following initial-boundary value problem for the Westervelt equation.
\begin{align}\label{eq:Westervelt}
   \left\{\begin{aligned}
      u_{tt} - c^2\Delta u - b \Delta u_t &= k (u^2)_{tt} &&\quad\text{in }J\times\Omega,\\
      u|_{\del\Omega} &= 0 &&\quad\text{in }J\times\del\Omega\\
      (u(0),u_t(0)) &= (u_0,u_1) &&\quad\text{in }\Omega.
   \end{aligned}\right.
\end{align}
Here $u:J\times\Omega\to\bbR$ is the unknown function and $(u_0,u_1):\Omega\to\bbR^2$ are the given initial data. In this paper we prove existence and uniqueness of global strong solutions for small initial data in anisotropic Sobolev-Slobodecki\u{\i} spaces:

\begin{thm}\label{thm:GLOBWP}
  Let $n\in\mathbb{N}$, $p>\max\{n/2,n/4+1\}$, $p\neq 3/2$ and suppose that the initial data $(u_0,u_1)$ satisfy the regularity and compatibility conditions
  \begin{align}\label{eq:RegData}
  \begin{split}
    u_0&\in \{v\in W^2_p(\Omega) : v|_{\del\Omega}=0 \},\\
  u_1&\in\begin{cases}W_p^{2-2/p}(\Omega)&,\ \text{if}\ p<3/2,\\
  					\{v\in W_p^{2-2/p}(\Omega): v|_{\partial\Omega}=0\}&,\ \text{if}\ p>3/2.\end{cases}
   \end{split}
   \end{align}
  Then for every $T\in(0,\infty]$ there is $\rho>0$ such that the smallness condition
  \begin{align*}
    \|u_0\|_{W^2_p}+\|u_1\|_{W^{2-2/p}_p}<\rho
  \end{align*}
  implies that problem \eqref{eq:Westervelt} admits a unique solution
  \begin{align}\label{eq:RegSol}
    u\in W^2_p(0,T;L_p(\Omega)) \cap W_p^1(0,T;W^2_p(\Omega)).
  \end{align}
  The solution map $(u_0,u_1) \mapsto u$ is a local isomorphism in these spaces.
\end{thm}

This is an extension of Theorems 1.1 and 1.2 of \cite{KaLa09}, where only the case $p=2$, $n\in\{1,2,3\}$ and $\|u_0\|_{W^2_2}+\|u_1\|_{W^2_2}<\rho$ is considered. Our result also implies that the regularity class in \cite[Theorem 1.1]{KaLa09} is not optimal. In order to provide optimal regularity results for the solutions of \eqref{eq:Westervelt}, we employ the powerful concept of maximal $L_p$-regularity (see Section 2). To this end we first rewrite \eqref{eq:Westervelt} as a quasilinear evolution equation, using a smallness condition like \cite[Assumption 2.2]{KaLa09} which guarantees that \eqref{eq:Westervelt} is a parabolic problem. Then we can apply the methods of the abstract parabolic theory, cf. \cite{Ama95,DHP03,DHP07,Pru02b}.

Trace theory implies that the values $u(t)$, $u_t(t)$ at $t\ge 0$ are well-defined and for each $t\ge 0$ the trace map $[u \mapsto (u(t), u_t(t))]$ is bounded from the class of solutions \eqref{eq:RegSol} into the class of the data \eqref{eq:RegData}. This means that the initial regularity is preserved for all $t>0$. Concerning the asymptotic behavior of solutions we prove that the equilibrium $u=0$ is exponentially stable. To be precise, we show the following result.
\begin{thm}\label{thm:ExpDec}
  Let $u$ be the solution in \prettyref{thm:GLOBWP}. Then there exist constants $C,\omega>0$ such that
  $$\|u(t)\|_{W^2_p} + \|u_t(t)\|_{W^{2-2/p}_p} \leq C e^{-\omega t} \left( \|u_0\|_{W_p^2}+\|u_1\|_{W_p^{2-2/p}} \right), \quad t\geq 0.$$
  Moreover, it holds that
  $$[t\mapsto e^{\omega t}u_t(t)]\in W^2_p(\delta,\infty;L_p(\Omega)) \cap W_p^1(\delta,\infty;W^2_p(\Omega))$$
  for each $\delta>0$ and
  $$e^{\omega t}\left( \|u_t(t)\|_{W^2_p} + \|u_{tt}(t)\|_{W^{2-2/p}_p} \right)\to 0, \quad\text{as}\quad  t\to\infty.$$
\end{thm}
This result is designed as a counterpart to \cite[Theorem 1.3]{KaLa09}. There, the authors prove that $\|u_{tt}(t)\|_{L_2} \leq C e^{-\omega t}(\|u_0\|_{W_2^2}+\|u_1\|_{W_2^2})$ for all $t\geq 0$, which is not possible in our case, since only $u_1\in W^{2-2/p}_p(\Omega)$ is assumed. However, we are able to improve the qualitative behavior of $(u,u_t,u_{tt})$ with respect to convergence in the natural trace spaces.

This paper is organized as follows. In Section 2, we reformulate \eqref{eq:Westervelt} as a quasilinear evolution equation of the form
\begin{align*}
  v_t + \calA_\#(v) v + \calB(v) v = F(v) \text{ in $J$}, \quad v(0)=v_0,
\end{align*}
where we have introduced new functions $v=[u,u_t]^{\sf T}$, $v_0=[u_0,u_1]^{\sf T}$. For sufficiently small functions $v_*$, we prove that the linearized problem
\begin{align*}
  v_t + \calA_\#(v_*) v + \calB(v_*) v = f \text{ in $J$}, \quad v(0)=v_0,
\end{align*}
has the property of maximal $L_p$-regularity even in exponentially time-weighted $L_p$-spaces, see Theorem \ref{thm:MaxReg_calA}.

In Section 3 we prove \prettyref{thm:GLOBWP} by applying the implicit function theorem in exponentially time-weighted $L_p$-spaces.

Section 4 is devoted to the qualitative behavior of the solutions of \eqref{eq:Westervelt}. The exponential stability of the equilibrium $u=0$ is a direct consequence of the implicit function theorem. The higher regularity of the solution $u$ in Theorem \ref{thm:ExpDec} follows from the parabolicity of the problem, whereas the exponential decay of $u_{tt}$ can be seen by differentiating the equation for $u$ and by studying a corresponding linear evolution equation for $u_{tt}$.

\subsection{Notations}

Let $p\in (1,\infty)$, $J=[0,T]$, $T>0$, and let $X$ be a Banach space. Then $L_p(J;X)$ denotes the Banach space of all $p$-integrable functions with values in $X$ and $W_p^k(J;X)$ denotes the vector-valued Sobolev space of order $k\in\mathbb{N}$. We will frequently use the notation $_0W_p^1(J;X)$, where
$$_0W_p^1(J;X):=\{u\in W_p^1(J;X):u(0)=0\}.$$ 
If $X$ and $Y$ are Banach spaces, then $\calL(X,Y)$ denotes the space of all bounded linear operators from $X$ to $Y$.

\section{Maximal $L_p$-regularity for the linearized problem}

Let $J=[0,T]$ or $J=\bbR_+:=[0,\infty)$ and let $p\in(1,\infty)$. A closed linear operator $A:D(A)\to X$ with dense domain $D(A)$ in a Banach space $X$ is said to admit \emph{maximal $L_p$-regularity} on $J$, if for each $f\in L_p(J;X)$ the abstract Cauchy problem
\begin{align*}
  u_t + A u(t) = f(t), \quad t\in J, \quad u(0)=0
\end{align*}
admits a unique solution $u\in \bbE(J):=W^1_p(J;X) \cap L_p(J;D(A))$, cf. \cite{Ama95,DHP03,Pru02b}. Here $D(A)$ is a Banach space for the graph norm $\|\cdot\|_X+\|A\cdot\|_X$. The space $\bbE(J)$ is continuously embedded into the space $BUC(J;\tr\bbE)$ of bounded and uniformly continuous functions on $J$ with values in the trace space
\begin{align*}
   \tr\bbE = D_A(1-1/p,p) := (X,D(A))_{1-1/p,p},
\end{align*}
where $(\cdot,\cdot)_{\theta,p}$ indicates real interpolation. We refer to \cite{AdFo03,Ama95,Tri83} for more information on these function spaces.

We say that the abstract inhomogeneous Cauchy problem
\begin{align}\label{eq:AICP}
  u_t + A u(t) = f(t), \quad t\in J, \quad u(0)=u_0
\end{align}
admits maximal $L_p$-regularity, if the solution map
\begin{align*}
  (f,u_0)\mapsto u, \quad L_p(J;X) \times \tr\bbE \to \bbE(J)
\end{align*}
is a topological isomorphism. Its inverse $u\mapsto (f,u_0)$ will be denoted by $(\del_t+A,\tr_{t=0})$, where $\tr_{t=0} : u \mapsto u(0)$ is the trace map. If $A$ has maximal $L_p$-regularity on $J$, then also the abstract inhomogeneous Cauchy problem \eqref{eq:AICP} has maximal $L_p$-regularity on $J$, see \cite{Ama95}, Section III.1.5 and \cite{Pru02b}, Proposition 1.2. If \eqref{eq:AICP} has maximal $L_p$-regularity then the following a priori estimate is valid.
\begin{align}\label{eq:a_priori}
  \|u\|_{\bbE(J)} \leq M(J) \left( \|u_t+A u\|_{L_p(J;X)} + \|u(0)\|_{\tr\bbE} \right), \quad u\in\bbE(J).
\end{align}
Here $M(J)<\infty$ is called a maximal $L_p$-regularity constant of $A$. It is independent of $J$ in the following sense. For any $T_0\in(0,\infty)$ there is $M_0<\infty$ such that \eqref{eq:a_priori} holds with $M(J)=M_0$ for any interval $J=[0,T]\subset[0,T_0]$. If $A$ has maximal $L_p$-regularity on $\bbR_+$ we may also take $T_0=\infty$.

Let us now reformulate problem \eqref{eq:Westervelt} as a quasilinear evolution equation for the new function $v=[u,u_t]^{\sf T}$. Taking into account that $(u^2)_{tt}=2u_{tt}u+2u_t^2$ we rewrite the first equation in \eqref{eq:Westervelt}
as
\begin{align*}
  v_t+ \begin{bmatrix} 0 & -I \\ 0 & -\frac{b}{1-2kv_1}\Delta \end{bmatrix} v + \begin{bmatrix} 0 & 0 \\ -\frac{c^2}{1-2kv_1}\Delta & 0 \end{bmatrix} v &= \begin{bmatrix}0\\\frac{2v_2^2}{1-2kv_1}\end{bmatrix},
\end{align*}
where $v=[v_1,v_2]^{\sf T}$. We now define
\begin{align*}
  \calA_\#(v):=\begin{bmatrix} 0 & -I \\ 0 & -\frac{b\Delta_D}{1-2kv_1}\end{bmatrix}, \quad \calB(v):=\begin{bmatrix} 0 & 0 \\ -\frac{c^2\Delta_D}{1-2kv_1} & 0 \end{bmatrix}, \quad F(v) := \begin{bmatrix}0\\\frac{2v_2^2}{1-2kv_1}\end{bmatrix},
\end{align*}
where $\Delta_D$ stands for the Laplace operator with homogeneous Dirichlet boundary conditions. This turns \eqref{eq:Westervelt} into the quasilinear initial value problem
\begin{align}\label{eq:Westervelt_2}
  v_t + \calA_\#(v) v + \calB(v) v = F(v) \text{ in $J$}, \quad v(0)=v_0.
\end{align}
To ensure that the problem \eqref{eq:Westervelt_2} is parabolic, we concentrate on solutions $v$ which are real-valued and bounded in the sense that
\begin{align}\label{eq:Bound}
  \sup \{ |v_1(t,x)| : t > 0, x\in\Omega \} < 1/(2k).
\end{align}
In the remaining part of this section we investigate the linear problem
\begin{align}\label{eq:Westervelt_3}
  v_t + \calA_\#(v_*) v + \calB(v_*) v = f \text{ in $J$}, \quad v(0)=v_0,
\end{align}
where $v_*$ is a fixed function satisfying \eqref{eq:Bound}, $v:J\times\Omega\to\bbR^2$ is the unknown function, $f:J\times\Omega\to\bbR^2$ and $v_0:\Omega\to\bbR^2$ are given functions and
\begin{align*}
  \calA_\#(v_*) &= \begin{bmatrix} 0 & -I \\ 0 & bA(v_*)\end{bmatrix}, \quad \calB(v_*) = \begin{bmatrix} 0 & 0 \\ c^2A(v_*) & 0 \end{bmatrix}, \\
  A(v_*)u &:= -(1-2k (v_*)_1)^{-1}\Delta_D u.
\end{align*}
To shorten the notation we will sometimes drop the dependence with respect to $v_*$. It is well-known that $A=A(v_*)$ has maximal $L_q$-regularity in $L_p(\Omega)$, see for instance \cite{DHP03,Dor93,HiPr97,KuWe04,LSU68}. We state this as 
\begin{thm}\label{thm:MaxReg_A}
  Let $p,q\in (1,\infty)$, $\Omega$ be a bounded $C^2$-domain in $\bbR^n$, $a\in C(\bar\Omega)$, $a(x)\geq a_0 >0$, ($x\in\Omega$). Then the operator $Au(x)=-a(x)\Delta u(x)$ with domain $D(A)=\{u\in W^2_p(\Omega) : u|_{\del\Omega}=0\}$ in $L_p(\Omega)$, admits maximal $L_q$-regularity on $\bbR_+$ for each $q\in(1,\infty)$.
\end{thm}
We conclude that $A$, $bA$, $c^2A$ have maximal $L_q$-regularity on $\bbR_+$ in $L_p(\Omega)$ for any $p,q\in(1,\infty)$. Hence also $A_\alpha:=\alpha+bA$ has this property for any $\alpha>0$. Consider $\calA_\#$, $\calB$ as operators in the Banach space $X_0$ with domain $X_1$, defined by
\begin{align*}
  X_0 := D(A)\times L_p(\Omega), \quad X_1 := D(A)\times D(A).
\end{align*}
Then it follows that $\calA_\#\in\calL(X_1,X_0)$ and $\calB\in\calL(X_0,X_0)$. The next goal is to show that
\begin{align*}
  \calA := \calA_\#+\calB = \begin{bmatrix}
    0 & -I \\ c^2 A & b A
  \end{bmatrix}.
\end{align*}
with domain $X_1$ in the space $X_0$ has maximal $L_q$-regularity on $\bbR_+$. To this purpose we first show in \prettyref{thm:A_principal_R-sectorial} that $\mu+\calA$ has maximal $L_q$-regularity on $\mathbb{R}_+$ in $X_0$ for some $\mu>0$. Then we use properties of the spectrum $\sigma(\calA)$ to prove that we may take $\mu=0$.

\begin{thm}\label{thm:A_principal_R-sectorial}
  Let $\alpha>0$ and $q\in (1,\infty)$. Then $\calA_{\#,\alpha}=\alpha +\calA_\# :X_1\to X_0$ has maximal $L_q$-regularity on $\bbR_+$ and there exists $\mu>0$ such that $\mu+\calA=\mu+\calA_{\#}+\calB$ enjoys the same property.
\end{thm}
\begin{proof}
Consider the system
\begin{equation}\label{eq:MR1}
v_t+\calA_{\#,\alpha}v=f,\quad v(0)=0\Blue{.}
\end{equation}
We have to show that for each $f\in L_q(\bbR_+;X_0)$ there exists a unique solution
$$v\in\!_0W_q^1(\bbR_+;X_0)\cap L_q(\bbR_+;X_1).$$
If $f=(f_1,f_2)$ and $v=(v_1,v_2)$, then \eqref{eq:MR1} is equivalent to
\begin{align*}\label{eq:MR2}
\partial_t v_1+\alpha v_1-v_2&=f_1,\\
\partial_t v_2+A_\alpha v_2&=f_2,
\end{align*}
subject to the initial conditions $v_j(0)=0$. Since $A_\alpha$ has maximal $L_q$-regularity on $\bbR_+$, we may solve the second equation to obtain a unique solution
$$v_2\in \!_0W_q^1(\bbR_+;L_p(\Omega))\cap L_q(\bbR_+;D(A)).$$
Inserting $v_2$ into the first equation we obtain the initial value problem
\begin{equation}\label{eq:MR3}
\partial_tv_1+\alpha v_1=g,\quad v_1(0)=0,
\end{equation}
where $g:=v_2+f_1\in L_q(\bbR_+;D(A))$. It is evident that for $\alpha>0$ the operator $(\partial_t+\alpha)$ with domain $D(\partial_t+\alpha)=\!_0W_q^1(\bbR_+;D(A))$ is invertible in $L_q(\bbR_+;D(A))$ and the unique solution of \eqref{eq:MR3} is given by
$$v_1(t)=\int_0^t e^{-\alpha(t-s)}g(s) ds,\quad t\ge 0.$$
This shows that $\calA_{\#,\alpha}$ has maximal $L_q$-regularity on $\bbR_+$.

We finally prove the second assertion by a perturbation argument (Proposition 4.3 and Theorem 4.4 in \cite{DHP03}): Since $(\calB-\alpha):X_0\to X_0$ is a bounded linear operator and $\alpha+\calA_\#$ has maximal $L_q$-regularity on $\bbR_+$, there exists $\mu>0$ such that the operator $\mu+\calA=\mu+(\alpha+\calA_\#)+(\calB-\alpha)$ has maximal $L_q$-regularity on $\bbR_+$.
\end{proof}

Next, we analyze the spectrum $\sigma(\calA)$. Recall that $A$ is defined by 
$$(Au)(x):=-a(x)\Delta u(x)\quad \text{for}\quad u\in D(A)=\{u\in W^2_p(\Omega) : u|_{\del\Omega}=0\}$$
and some $a\in C(\bar{\Omega})$ with $a(x)\geq a_0 >0$ for all $x\in\Omega$. Since $A$ has compact resolvent, the spectrum $\sigma(A)$ is a discrete subset of $\mathbb{C}$ and it consists solely of eigenvalues with finite multiplicity. Let $\lambda\in\sigma(A)$ and $u\in D(A)$ be a corresponding eigenfunction. Then
$$(Au)(x)=\lambda u(x)\Leftrightarrow -a(x)\Delta u(x)=\lambda u(x)\Leftrightarrow -\Delta u(x)=\frac{\lambda}{a(x)}u(x).$$
Testing the last equation with $\bar{u}$ and integrating by parts ($u|_{\partial\Omega}=0$) it follows that $\lambda$ is real and $\lambda>0$ by the Poincar\'{e} inequality. We define
$$\lambda_1(A):=\min\{\lambda:\lambda\in\sigma(A)\}>0,$$
hence $\sigma(A)\subset[\lambda_1(A),\infty)$. In particular, $A:D(A)\to L_p(\Omega)$ is an isomorphism. Suppose that $\lambda$ belongs to the resolvent set $\rho(\calA)$ of $\calA$. Then the following resolvent formula holds.
\begin{align*}
  (\lambda-\calA)^{-1} &= \begin{bmatrix}
    - R_\lambda (\lambda - b A) & R_\lambda \\
    I + \lambda R_\lambda (\lambda - bA) & -\lambda R_\lambda
  \end{bmatrix}.
\end{align*}
Here $R_\lambda$ is defined by $R_\lambda := (-\lambda^2 I + (\lambda b - c^2) A)^{-1}$. If $\lambda b \neq c^2$ then
\begin{align*}
  R_\lambda = - \frac{1}{\lambda b - c^2} \left( \frac{\lambda^2}{\lambda b - c^2} I - A \right)^{-1}.
\end{align*}
This useful identity yields the following characterization of $\rho(\calA)$.

\begin{lem}\label{lem:Resolvent}
  $\lambda\in\rho(\calA)$ if and only if $\lambda b - c^2 \neq 0$ and $\mu(\lambda):=\frac{\lambda^2}{\lambda b - c^2} \in \rho(A)$.
\end{lem}
\begin{proof}
  Let $\lambda\in\rho(\calA)$. Then $R_\lambda:L_p(\Omega)\to D(A)$ must be bounded. Suppose that $\lambda b - c^2 = 0$. Then $(-\lambda^2 I)^{-1} : L_p(\Omega) \to D(A)$ would be bounded, but that is not possible. Therefore we must have $\lambda b - c^2 \neq 0$ and this implies $\frac{\lambda^2}{\lambda b - c^2} \in \rho(A)$.

  Conversely, let $\lambda b - c^2 \neq 0$ and $\frac{\lambda^2}{\lambda b - c^2} \in \rho(A)$. Then $R_\lambda: L_p(\Omega)\to D(A)$ is bounded and we can easily check that $(\lambda-\calA)^{-1} : X_0\to X_1$ is bounded.
\end{proof}

\begin{lem}\label{lem:spectral_bound}
  If $\Re\lambda<\lambda_0:=\min\{\frac b 2 \lambda_1(A), \frac{c^2}b\}$ then $\lambda\in\rho(\calA)$. 
\end{lem}
\begin{proof}
  We prove the assertion by contraposition. Let $a\geq\lambda_1(A)$ and let $\lambda\in\bbC$ satisfy $a=\mu(\lambda)=\lambda^2/(\lambda b - c^2)$. Then $\lambda$ has the form $\frac{ab}2 \pm \sqrt{\frac{a^2b^2}4-ac^2}$. In the case $\frac{a^2b^2}4-ac^2\leq 0$ we obtain $\Re\lambda=\frac{ab}2\geq\frac b 2 \lambda_1(A)$. In the other case $\frac{a^2b^2}4-ac^2\geq 0$ we obtain that $\lambda$ is real and  \begin{align*}
    \lambda \geq \frac{ab}2\left(1-\sqrt{1-\frac{4c^2}{ab^2}}\right) = \frac{ab}2 \frac {1-\left(1-\frac{4c^2}{ab^2}\right)} {1+\sqrt{1-\frac{4c^2}{ab^2}}} = \frac {\frac{2c^2}b} {1+\sqrt{1-\frac{4c^2}{ab^2}}} \geq \frac {c^2}b.
  \end{align*}
  We have shown that $\mu(\lambda)\geq\lambda_1(A)$ implies $\Re\lambda\geq \lambda_0$. Consequently, if $\Re\lambda < \lambda_0$, then $\mu(\lambda)\in\rho(A)$ since $\sigma(A)\subset[\lambda_1(A),\infty)$. By \prettyref{lem:Resolvent} this implies $\lambda\in\rho(\calA)$.

\end{proof}

Having analyzed the spectrum for the operator $\calA$, we can continue to prove maximal $L_p$-regularity for this operator on $\mathbb{R}_+$. To this end, we consider again the linear problem \eqref{eq:Westervelt_3}. Concerning regularity of initial data $v_0$ we consider the space
\begin{align*}
  X_\gamma \Blue{:}= \left( X_0,X_1\right)_{1-1/p,p} &= D(A) \times D_A(1-1/p,p).
\end{align*}
Note that the real interpolation space $D_A(1-1/p,p)$ can be computed to the result
$$D_A(1-1/p,p)=\begin{cases}W_p^{2-2/p}(\Omega)&,\ \text{if}\ p<3/2,\\
  					\{u\in W_p^{2-2/p}(\Omega): u|_{\partial\Omega}=0\}&,\ \text{if}\ p>3/2,\end{cases}$$
see e.g.\ \cite{Gri69}. In case $p=3/2$ the space $D_A(1-1/p,p)$ looks more complicated.

Let $Y$, $\bbX(J)$ be Banach spaces such that $\bbX(J)\hookrightarrow L_{1,\text{loc}}(J;Y)$ and let $\omega\in\bbR$. To describe exponential decay of solutions we employ the Banach space
\begin{align*}
  e^{-\omega}\bbX(J) := \{u\in L_{1,\text{loc}}(J;Y): [t\mapsto e^{\omega t}u(t)] \in \bbX(J)\},
\end{align*}
equipped with the norm $\|u\|_{e^{-\omega}\bbX(J)} := \|[t\mapsto e^{\omega t}u(t)]\|_{\bbX(J)}$.

\begin{thm}\label{thm:MaxReg_calA}
  Let $p\in (1,\infty)$, $-\lambda_0 = \sup\{\Re\lambda : \lambda\in\sigma(-\calA)\}<0$ denote the spectral bound of $-\calA$ and let $\omega\in[0,\lambda_0)$. Then $\calA$ has maximal $L_p$-regularity on $\bbR_+$ in the sense that
   \begin{align*}
     (\del_t+\calA,\tr_{t=0}) : e^{-\omega}(W^1_p(\bbR_+;X_0) \cap L_p(\bbR_+;X_1)) \to e^{-\omega}L_p(\bbR_+;X_0) \times X_\gamma
   \end{align*}
   is an isomorphism.
\end{thm}
\begin{proof}
   By \prettyref{thm:A_principal_R-sectorial} the operator $\mu+\calA$ has maximal $L_p$-regularity on $\bbR_+$, hence $\calA$ has maximal $L_p$-regularity on bounded intervals, which can be seen by multiplying the equation $\del_t v + \calA v = f$ with $e^{-\mu t}$. By \prettyref{lem:spectral_bound} the spectral bound of $-\calA$ is strictly negative and equals $-\lambda_0$ and therefore the spectral bound of $-\calA+\omega$ equals $\omega-\lambda_0$ and it is strictly negative if $\omega\in[0,\lambda_0)$. Then \cite[Theorem 2.4]{Dor93} implies that $\calA-\omega$ has maximal $L_p$-regularity on $\bbR_+$ for each $\omega\in[0,\lambda_0)$, that is
   \begin{align*}
     (\del_t+\calA-\omega,\tr_{t=0}) : W^1_p(\bbR_+;X_0) \cap L_p(\bbR_+;X_1) \to L_p(\bbR_+;X_0) \times X_\gamma
   \end{align*}
   is an isomorphism. The assertion follows from \cite[Proposition III.1.5.3]{Ama95}.
\end{proof}

\section{Global wellposedness for the nonlinear problem}

In this section we prove a slightly more general form of \prettyref{thm:GLOBWP}, which we formulate for the transformed system \eqref{eq:Westervelt_2}:
\begin{thm}\label{thmLWP}
  Let $p>\max\{n/2,n/4+1\}$, $\omega\in[0,\lambda_0)$ and $J=\bbR_+$. Then there exists $\rho>0$ such that for all $v_0\in X_\gamma$ with $\|v_0\|_{X_\gamma}<\rho$ the problem
  \begin{equation}\label{LWP_Wes1}
  v_t + \calA(v)v= F(v),\ t\in J,\quad v(0)=v_0
  \end{equation}
 admits a unique solution
  \begin{align*}
    v=(v_1,v_2)\in e^{-\omega}\bbE(J) := e^{-\omega}(W^1_p(J;X_0) \cap L_p(J;X_1)),
  \end{align*}
such that $\|v_1\|_{BUC(J\times\Omega)}<m$ for some $m<1/(2k)$. The solution map $v_0\mapsto v$ is a local isomorphism.
\end{thm}
\prettyref{thm:GLOBWP} corresponds to the special case $\omega=0$. For the proof we will employ the following mapping property of substitution operators, which follows from the definition of the Fr{\'e}chet-derivative and the fundamental theorem of calculus.
\begin{lem}\label{lem:Subst}
  Let $k\in\{2,3,\ldots\}\cup\{\infty\}$, $T\in(0,\infty)$, $J=[0,T]$, and let $X$, $Y$ be Banach spaces and let $U$ be open in $X$. Suppose that $S\in C^k(U;Y)$ and each derivative $S^{(j)}$, with $j=2,3,\ldots k$, maps bounded sets into bounded sets. Then the substitution operator $\tilde S:=[u\mapsto S\circ u]$ belongs to the class
  $$C^{k-1}(BUC(J;U);BUC(J;Y)).$$
\end{lem}

\begin{proof}[Proof of Theorem \ref{thmLWP}]
%
{\bf Step 1:} In this step we prove that 
$$(\calA,F)\in C^\infty(V;\calL(X_1,X_0)\times X_0),$$
where $V \subset X_\gamma$ is defined as follows. Fix $m\in(0,(2k)^{-1})$ and let
$$V:=\left\{w=(w_1,w_2)\in X_\gamma:\|w_1\|_\infty<m\right\},$$
then $V$ is open in $X_\gamma$, since $X_\gamma\hookrightarrow C(\overline{\Omega})\times D_A(1-1/p,p)$, provided $p>n/2$. 

First of all, we write $F$ in the form $F(v_1,v_2)=(0,G_1(v_1)\cdot G_2(v_2))$, where $G_1(v_1)=(1-2kv_1)^{-1}$, $G_2(v_2)=2v_2^2$ and "$\cdot$" denotes pointwise multiplication. Then the maps
\begin{align*}
  G_1 : \{u\in L_\infty(\Omega): \|u\|_{L_\infty(\Omega)}<(2k)^{-1} \} \to L_\infty(\Omega), \quad G_2 : L_{2p}(\Omega) \to L_p(\Omega)
\end{align*}
are $C^\infty$-maps. Moreover, by H\"{o}lder's inequality, the pointwise multiplication is bilinear and continuous (and therefore $C^\infty$) from $L_\infty(\Omega) \times L_p(\Omega)$ to $L_p(\Omega)$. Finally, the injection of $V \subset W^2_p(\Omega) \times W^{2-2/p}_p(\Omega)$ into $L_\infty(\Omega)\times L_{2p}(\Omega)$ is linear and bounded, provided that $p>n/2$ and $p>1+n/4$ by Sobolev's embedding theorem. This proves that $F\in C^\infty(V;X_0)$.

Next, we decompose
\begin{align}\label{LWP_Wes3}
  \calA(v) = \calA_1(v) \calA_2 := \begin{bmatrix} 1 & 0 \\ 0 & \frac{1}{1-2k v_1} \end{bmatrix} \begin{bmatrix} 0 & -I \\ c^2 A & b A \end{bmatrix}, \quad v=(v_1,v_2)\in V.
\end{align}
Here the second matrix $\calA_2$ belongs to $C^\infty(V;\calL(X_1,X_0))$ since it is independent of $v$ and $A:D(A)\to L_p(\Omega)$ is linear and bounded. The mapping $[v_1\mapsto \frac{1}{1-2kv_1}]$ is $C^\infty$ from $\{v_1\in L_\infty(\Omega): \|v_1\|_{L_\infty(\Omega)}<(2k)^{-1} \}$ into $L_\infty(\Omega)$ and therefore $[(v_1,w_2)\mapsto \frac{1}{1-2k v_1} w_2]$ is $C^\infty$ from $\{v_1\in L_\infty(\Omega): \|v_1\|_{L_\infty(\Omega)}<(2k)^{-1} \} \times L_p(\Omega)$ to $L_p(\Omega)$. Then the continuous embedding of $W^2_p(\Omega)$ into $L_\infty(\Omega)$ for $p>n/2$ implies $\calA_1\in C^\infty(V;\calL(X_0,X_0))$. We conclude that $\calA\in C^\infty(V;\calL(X_1,X_0))$.

\noindent  {\bf Step 2:} In this step we consider the substitution operators $\tilde F : u\mapsto F\circ u$ and $\tilde \calA : u \mapsto \calA\circ u$ for time-dependent functions $u$ on $J\times\Omega$ and prove that they are also $C^\infty$ in appropriate spaces. Let
$$i:e^{-\omega}(W^1_p(J;X_0) \cap L_p(J;X_1))\to e^{-\omega}BUC(J;X_\gamma)$$
	be the inclusion map by \cite[Theorem III.4.10.2. \& Remark III.4.10.9.(a)]{Ama95}. Since $e^{-\omega}BUC(J;V)$ is open in $e^{-\omega}BUC(J;X_\gamma)$ it follows that
$$W:=i^{-1}\left(e^{-\omega}BUC(J;V)\right)\subset e^{-\omega}(W^1_p(J;X_0) \cap L_p(J;X_1))=e^{-\omega}\bbE(J),$$
is open. We show that $\tilde{F}\in C^\infty(W;e^{-\omega}L_p(J;X_0))$. Again we use the decomposition $\tilde F(v_1,v_2) = (0, \tilde G_1(v_1)\cdot \tilde G_2(v_2))$ with $\tilde G_1(v_1)=(1-2kv_1)^{-1}$, $\tilde G_2(v_2)=2v_2^2$ and prove that $\tilde{F}$ is the composition of $C^\infty$-maps. First observe that the injections
\begin{align*}
  e^{-\omega}\bbE(J) \hookrightarrow e^{-\omega} L_\infty(J\times\Omega) \times \left( e^{-\omega} L_\infty(J;L_{2p}(\Omega)) \cap e^{-\omega} L_p(J;L_{2p}(\Omega)) \right)
\end{align*}
and $e^{-\omega} L_\infty(J\times\Omega) \hookrightarrow L_\infty(J\times\Omega)$ are continuous and thus $C^\infty$. Then, similarly as in Step 1, we conclude that the map $\tilde G_1 : v_1 \mapsto (1-2k v_1)^{-1}$,
\begin{align*}
  \tilde G_1 : \left\{ v_1\in L_\infty(J\times\Omega) : \|v_1\|_{L_\infty(J\times\Omega)}<(2k)^{-1} \right\} \to L_\infty(J\times\Omega)
\end{align*}
is $C^\infty$. H\"{o}lder's inequality implies that pointwise multiplication
\begin{align*}
  (u,v)\mapsto uv : e^{-\omega} L_\infty(J;L_{2p}(\Omega)) \times e^{-\omega} L_p(J;L_{2p}(\Omega)) \to e^{-\omega} L_p(J;L_p(\Omega))
\end{align*}
is continuous. We conclude that
\begin{align*}
  \tilde G_2 : v_2\mapsto 2v_2^2, \; e^{-\omega} L_\infty(J;L_{2p}(\Omega)) \cap e^{-\omega} L_p(J;L_{2p}(\Omega)) \to e^{-\omega} L_p(J\times\Omega)
\end{align*}
is $C^\infty$. Since pointwise multiplication is also continuous as a mapping
\begin{align*}
  L_\infty(J\times\Omega) \times e^{-\omega} L_p(J\times\Omega) \to e^{-\omega} L_p(J\times\Omega),
\end{align*}
we obtain that $\tilde F : W \to e^{-\omega} L_p(J;X_0)$ is $C^\infty$.

In Step 1 we have shown that $\calA \in C^\infty(V;\calL(X_0,X_1))$. The previously used decomposition of $\calA$ implies that all its derivatives are bounded on $V$. From \prettyref{lem:Subst} and the embedding $e^{-\omega}BUC(J;X_\gamma)\hookrightarrow BUC(J;X_\gamma)$ we infer that
$$\calA\in C^\infty\left(e^{-\omega}BUC(J;V);BUC(J;\calL(X_1,X_0))\right).$$
Here we use the same notation for $\calA$ and the substitution operator induced by $\calA$. Furthermore it is easy to see that the mapping
$$[(B,v)\mapsto Bv]:BUC(J;\calL(X_1,X_0))\times e^{-\omega}L_p(J;X_1)\to e^{-\omega}L_p(J;X_0)$$
is continuous and bilinear, hence $C^\infty$. Defining $G(v):=\calA(v)v$ it follows that $G\in C^\infty(W;e^{-\omega}L_p(J;X_0))$.

\noindent {\bf Step 3:} Define a mapping $H:W\times X_\gamma\to e^{-\omega}L_p(J;X_0)\times X_\gamma$ by
$$H(v,v_0):=(v_t+\calA(v) v-F(v),\tr_{t=0}v-v_0).$$
By Step 2, we may conclude that
$$H\in C^1(W\times X_\gamma;e^{-\omega}L_p(J;X_0)\times X_\gamma),$$
since $\partial_t$ and $\tr_{t=0}$ are linear operators on $W$.

 Note that $v\in W$ is a solution of \eqref{LWP_Wes1} if and only if $H(v,v_0)=0$. Obviously it holds that $H(0,0)=0$ and
 $$D_vH(0,0)w= (w_t+\calA(0) w,\tr_{t=0}w),\quad w\in W.$$
 By Theorem \ref{thm:MaxReg_calA} the operator
 $$(\partial_t+\calA(0),\tr_{t=0}):e^{-\omega}(W^1_p(J;X_0) \cap L_p(J;X_1))\to e^{-\omega}L_p(J;X_0)\times X_\gamma$$
 is an isomorphism with bounded inverse, hence by the implicit function theorem there exists an open ball $\bbB_\rho(0)\subset X_\gamma$ and a unique function $\phi\in C^1(\bbB_\rho(0);W)$ such that $H(\phi(v_0),v_0)=0$ for all $v_0\in \bbB_\rho(0)$ and $\phi(0)=0$. Therefore $v:=\phi(v_0)\in W$ is the unique solution of \eqref{LWP_Wes1} and $i(v)\in e^{-\omega}BUC(J;V)$, hence $\|v_1(t)\|_\infty<m<1/(2k)$ for all $t\in J$. \qedhere

\end{proof}

\section{Exponential Stability}

Let
$$\bbE(\mathbb{R}_+):=W^1_p(\bbR_+;X_0) \cap L_p(\bbR_+;X_1).$$
By Theorem \ref{thmLWP} with $T=\infty$ there exists $\rho>0$ and a function $\phi$ in the class $C^1(\bbB_\rho(0);e^{-\omega}\bbE(\bbR_+))$ with $\phi(0)=0$ such that $v=\phi(v_0)\in e^{-\omega}\bbE(\bbR_+)$ is the unique solution of \eqref{LWP_Wes1}.
Then we obtain
$$\phi(v_0)=\int_0^1\frac{d}{d\theta} \phi(\theta v_0)\, d\theta=\left(\int_0^1\phi'(\theta v_0)d\, \theta\right) v_0,$$
for all $v_0\in\bbB_\rho(0)$. Since $\phi$ is $C^1$, we see that for each $\eps>0$ there exists $\rho_0=\rho_0(0)>0$ such that
$$\|\phi'(v_0)-\phi'(0)\|_{\calL(X_\gamma;e^{-\omega}\bbE(\bbR_+))}\le\eps,$$
whenever $\|v_0\|_{X_\gamma}\le\rho_0<\rho$. It follows that $\phi'(v_0)$ is uniformly bounded in $\calL(X_\gamma;e^{-\omega}\bbE(\bbR_+))$ w.r.t. $v_0$ if $\|v_0\|_{X_\gamma}\le\rho_0<\rho$. Therefore we obtain the estimate
$$e^{\omega t}\|v(t)\|_{X_\gamma}\le C_\gamma\|\phi(v_0)\|_{e^{-\omega}\bbE(\bbR_+)}\le C\|v_0\|_{X_\gamma},$$
which is valid for all $t\ge 0$, $|v_0|_{X_\gamma}\le\rho_0<\rho$ and $C>0$ does not depend on $v_0$ and $t$. This implies exponential stability of the equilibrium $0$ in $X_\gamma$ and the first estimate in \prettyref{thm:ExpDec}. If $\omega=0$ then we still have $\|v(t)\|_{X_\gamma}\to 0$ as $t\to \infty$, since $v\in\bbE(\bbR_+)\hookrightarrow C_0(\bbR_+;X_\gamma)$.

But we can show even more. By \cite[Theorem 5.1]{Pru02b} we obtain that
\begin{equation}\label{EXS_Wes0}
v\in W_p^2(\delta,T;X_0)\cap W_p^1(\delta,T;X_1),
\end{equation}
for all $0<\delta<T<\infty$, since $(\calA,F)\in C^1(V;\calL(X_1,X_0)\times X_0)$ and since $\calA(v(t))$ has maximal $L_p$-regularity for each fixed $t\in [0,T]$, by Theorem \ref{thm:MaxReg_calA} (recall that $\|v_1(t)\|_{\infty}<m<1/(2k)$ for all $t\ge 0$). In particular, \eqref{EXS_Wes0} implies that $v_t(t)\in X_\gamma$ for each $t>0$.

We will show that we may even set $T=\infty$ in \eqref{EXS_Wes0} and that $e^{\omega t}v_t(t)\to 0$ in $X_\gamma$ as $t\to\infty$. The result on higher parabolic regularity enables us to differentiate equation \eqref{LWP_Wes1} w.r.t. $t>\delta$ to the result
$$v_{tt}(t)+[\calA'(v(t))v_t(t)]v(t)+\calA(v(t))v_t(t)=F'(v(t))v_t(t),\quad t>\delta.$$
We now distinguish between the function $\bar{w}:=v_t$ and the fixed solution $v$, hence $\bar w$ is a solution to the linear problem
\begin{equation}\label{EXS_Wes1}
\bar{w}_t+\calA(0) \bar{w}(t) =B(t) \bar{w}(t) , \quad t>\delta, \quad \bar{w}(\delta)=v_t(\delta),
\end{equation}
where
$$B(t)\bar{w}(t) :=F'(v(t)) \bar{w}(t)  +[\calA(0)-\calA(v(t))] \bar{w}(t)-[\calA'(v(t)) \bar{w}(t) ]v(t).$$
The operator $B(t)$ is well-defined, since $v(t)\in X_1=D(A)\times D(A)$ for a.a. $t>0$.

We claim that for each $\eps>0$ there exists a sufficiently large $\delta>0$ such that the estimate
\begin{equation}\label{EXS_Wes1a}
\|B(\cdot)w\|_{e^{-\omega}L_p(\delta,\infty;X_0)}\le\eps\|w\|_{e^{-\omega}\bbE(\delta,\infty)}
\end{equation}
holds for each $w\in e^{-\omega}{_0\bbE(\delta,\infty)}$, where $\0\bbE(\delta,\infty):=\{u\in\bbE(\delta,\infty) : u(\delta)=0 \}$. Since $e^{\omega t}v(t)\to 0$ in $X_\gamma$ as $t\to\infty$ and $v(t)\in V$ for all $t\ge 0$, the desired estimate follows for the first and the second term, so we concentrate on the third one. We use that the spaces $e^{-\omega}{\bbE(\delta,\infty)}$ are continuously embedded into $e^{-\omega}BUC(\delta,\infty;X_\gamma)$ and the embedding constant is independent of $\delta$, which follows from identities like
\begin{align*}
  \|u\|_{e^{-\omega}\bbE(\delta,\infty)} &= \|[t\mapsto e^{\omega(t+\delta)}u(t+\delta)]\|_{\bbE(0,\infty)}.
\end{align*}
From \eqref{LWP_Wes3} we obtain
$$\|[\calA'(v)w]v\|_{e^{-\omega}L_p(\delta,\infty;X_0)}=2k\left\|\frac{w_1}{(1-2kv_1)^2}(c^2Av_1+bAv_2)\right\|_{e^{-\omega}L_p(\delta,\infty;L_p(\Omega))}.$$
Since $\|v_1(t)\|_\infty<m<1/(2k)$ for all $t\ge 0$ and $w\in e^{-\omega}{_0\bbE(\delta,\infty)}\hookrightarrow e^{-\omega}BUC(\delta,\infty;X_\gamma)$ and $D(A)\hookrightarrow L_\infty(\Omega)$ it follows that
\begin{align*}
&\|[\calA'(v)w]v\|_{e^{-\omega}L_p(\delta,\infty;X_0)}\\
&\quad\le C_m\|w\|_{e^{-\omega}BUC(\delta,\infty;X_\gamma)}\|(c^2Av_1+bAv_2)\|_{e^{-\omega}L_p(\delta,\infty;L_p(\Omega))}\\
&\quad\le C_m\|w\|_{e^{-\omega}\bbE(\delta,\infty)}\|v\|_{e^{-\omega}L_p(\delta,\infty;X_1)}.
\end{align*}
Since $v\in e^{-\omega}\bbE(\bbR_+)$ is the solution of \eqref{LWP_Wes1} and thereby a fixed function, it follows that $\|v\|_{e^{-\omega}L_p(\delta,\infty;X_1)}\le \varepsilon /C_m$, provided that $\delta>0$ is sufficiently large. This yields the claim.

Finally we write
$$\partial_t+\calA(0)-B(t)=[I-B(t)(\partial_t+\calA(0))^{-1}](\partial_t+\calA(0)).$$
We show that the operator on the left side is an isomorphism from $e^{-\omega} {_0\bbE(\delta,\infty)}$ to $e^{-\omega}L_p(\delta,\infty;X_0)$. To this end, it suffices to show that the first operator on the right hand side is an isomorphism from $e^{-\omega}L_p(\delta,\infty;X_0)$ to $e^{-\omega}L_p(\delta,\infty;X_0)$. We have
\begin{align*}
\|B(\cdot)(\partial_t+\calA(0))^{-1}h\|_{e^{-\omega}L_p(\delta,\infty;X_0)}&\le\eps\|(\partial_t+\calA(0))^{-1}h\|_{e^{-\omega}\bbE(\delta,\infty)}\\
&\le\eps C\|h\|_{e^{-\omega}L_p(\delta,\infty;X_0)}.
\end{align*}
Here the maximal regularity constant $C>0$ of $(\partial_t+\calA(0))$ does not depend on $\delta\in\bbR_+$. To see this, let $M_\omega$ denote the norm of $(\del_t-\omega+\calA(0))^{-1}:L_p(\bbR_+;X_0)\to\bbE(\bbR_+)$. Using a translation from $(\delta,\infty)$ to $(0,\infty)$ we obtain
\begin{align*}
  \|u\|_{e^{-\omega}\bbE(\delta,\infty)} &= \|[t\mapsto e^{\omega(t+\delta)}u(t+\delta)]\|_{\bbE(\bbR_+)}\\
  &\leq M_\omega \|[t\mapsto (\del_t-\omega+\calA(0)) (e^{\omega(t+\delta)}u(t+\delta))] \|_{L_p(\bbR_+;X_0)} \\
  &= M_\omega \|(\del_t+\calA(0)) u\|_{e^{-\omega}L_p(\delta,\infty;X_0)}.
\end{align*}
If we choose $\eps C<1$, hence $\delta>0$ sufficiently large, a Neumann series argument implies that the operator $\partial_t+\calA(0)-B(\cdot)$ is invertible from $e^{-\omega}{_0\bbE(\delta,\infty)}$ to $e^{-\omega}L_p(\delta,\infty;X_0)$ with bounded inverse.

From now on we fix such a $\delta>0$, say $\delta^*>0$, and define $w(\delta^*):=v_t(\delta^*)\in X_\gamma$. By Theorem \ref{thm:MaxReg_calA} the function $w_*(t):=e^{-\calA(0) (t-\delta^*)}w(\delta^*)$ belongs to $e^{-\omega} {\bbE(\delta^*,\infty)}$. Then the shifted function $\tilde{w}:=\bar{w}-w_*=v_t-w_*$ solves the problem
\begin{equation}\label{EXS_Wes2}
\tilde{w}_t(t) +\calA(0)\tilde{w}(t) -B(t)\tilde{w}(t) =B(t)w_*(t) ,\quad t>\delta^*,\quad \tilde{w}(\delta^*)=0.
\end{equation}
It is easily seen that $B(\cdot)w_*$ is in $e^{-\omega}L_p(\delta^*,\infty;X_0)$. By the preceding arguments it follows that $\tilde{w}\in e^{-\omega}{_0\bbE(\delta^*,\infty)}$ or equivalently $v_t\in e^{-\omega}{\bbE(\delta^*,\infty)}$. Property \eqref{EXS_Wes0} with $T=\delta^*$ thus yields  $v_t\in e^{-\omega}{\bbE(\delta,\infty)}$ for \emph{each} $\delta>0$ (the exponential weight plays no role on the bounded interval $(\delta,\delta^*)$). This implies in particular that $e^{\omega t} v_t(t)\to 0$ in $X_\gamma=D(A)\times D_A(1-1/p,p)$ as $t\to\infty$. We summarize these results in
\begin{thm}
 Let $n\in\mathbb{N}$, $p>\max\{n/2,n/4+1\}$, $\omega\in [0,\lambda_0)$ and let $v_0$ satisfy the conditions of Theorem \ref{thmLWP}. Then the equilibrium $v_*=0$ of \eqref{LWP_Wes1} is stable in $X_\gamma$ and $e^{\omega t}\|v(t)\|_{X_\gamma}\to 0$ as $t\to \infty$. Moreover, it holds that $v_t\in e^{-\omega}\bbE(\delta,\infty)$ for each $\delta>0$, hence $e^{\omega t}\|v_t(t)\|_{X_\gamma}\to 0$ as $t\to \infty$.
\end{thm}

\bibliographystyle{plain}
\bibliography{MeWi11_Lit}

\end{document}